\documentclass[a4paper,10pt]{article}

\usepackage{amsmath}

\usepackage{amsthm}
\usepackage[dvips]{graphicx}

\usepackage{epsfig}

\usepackage{psfrag}

\usepackage{srcltx}

\usepackage{amssymb}
\usepackage{latexsym}

\newtheorem{Theorem}{Theorem}

\newtheorem{Definition}[Theorem]{Definition}
\newtheorem{Proposition}[Theorem]{Proposition}
\newtheorem{Lemma}[Theorem]{Lemma}

\newtheorem{Remark}[Theorem]{Remark}

\begin{document}

\title{Hyperbolic polygons of minimal perimeter with given angles}

\author{Joan Porti\footnote{Partially supported by the 
Spanish Micinn/FEDER through grant MTM2009-07594 and prize ICREA ACADEMIA 2008}} 
\date{\today}

\maketitle

\begin{abstract}
We prove that, among all convex hyperbolic polygons with given angles,
the perimeter is minimized by the unique polygon with an inscribed circle.
The proof relies on work of J.-M.\ Schlenker \cite{Schlenker}.
\end{abstract}

Let $0<\beta_1,\ldots,\beta_n<\pi$ be a finite set of angles, and consider $\mathcal P$ 
the set of all compact convex hyperbolic polygons with given ordered angles $0<\beta_1,\ldots,\beta_n<\pi$.
Along the paper, the edges shall be ordered counter-clockwise.
 Assume that 
$\sum (\pi-\beta_i)> 2\pi$, so that $\mathcal P\neq\emptyset$.
The purpose of this note is to give a proof of the following theorem.

\begin{Theorem}
\label{theorem1}
The unique minimum of the perimeter in $\mathcal P$ is realized by the polygon with an inscribed circle.
\end{Theorem}

Recall that a circle is \emph{inscribed} in a polygon if it is tangent to all of its edges.

The motivation to consider Theorem~\ref{theorem1} comes from the paper \cite{Porti} where it is proved for polygons with angles $\leq \pi/2$, 
as a consequence of the proof of a regeneration result of hyperbolic cone three-manifolds. 
Having not found it in the literature,   here
we give  an easier proof   using only 
tools from plane hyperbolic geometry,  de Sitter sphere and some of the ideas of \cite{Schlenker}.

\medskip

The spherical and Euclidean analogs of Theorem~\ref{theorem1} are know.
By using the polar, the spherical version is a corollary
of Theorem~\ref{thm:fixedlength}(i) below, 
 due to Steiner \cite{Steiner}. 
The Euclidean one is due to  Lindel\"of \cite{lindelof}  (notice that  the area must be fixed to avoid homotheties):

\begin{Theorem}[\cite{Steiner},\cite{lindelof}]
Among all convex spherical polygons with given angles, and among all convex Euclidean polygons with given angles and fixed area, 
the perimeter is minimized by the polygon with an inscribed circle. 
\end{Theorem}

Lindel\"of proves the analog for polyhedra in Euclidean space, with given normal directions for the faces. The proof is also found in 
Alexandrov
 \cite[\S 8.2]{Alexandrov} for polyhedra, and in Knebelman \cite{Knebelman} for polygons.
In the hyperbolic setting, there is not a clear analog for polyhedra. See \cite{Rivin,Schlenker,Thurston} and  references therein for convexity results of hyperbolic polyhedra.

\medskip

The proof of  Theorem~\ref{theorem1} relies on 
 the techniques and results of Schlenker in \cite{Schlenker}, by using the polar   in the de Sitter sphere,
though the presentation here is self-contained.  
In particular it is essentially the de Sitter analog of the following theorem proved in \cite{Schlenker}:

\begin{Theorem}[\cite{Schlenker,Steiner}]
\label{thm:fixedlength} 
 In a plane of constant curvature, among all polygons  with given edge lengths, the area is maximized by:
      \begin{itemize}
                 \item[(i)] The polygon  with a circumscribed circle in the spherical    and Euclidean plane.
		\item[(ii)] The polygon with vertices contained in a curve of constant principal curvature (namely a circle, a horocycle, or 
the equidistant of a geodesic) in the hyperbolic plane. 
        \end{itemize}
 \end{Theorem}

The spherical version is due to Steiner \cite{Steiner}, but the  Euclidean one was known before. The hyperbolic one is due to Schlenker \cite{Schlenker}, who reproves all the cases
with a unified approach.

\medskip

In Section~\ref{section:Lorentz} some material about Lorentz space and de Sitter sphere is recalled. Section~\ref{section:spaceofpolygons}
is devoted to the proof of
Theorem~\ref{theorem1}.

\section{Lorentz  space and de Sitter sphere.}
\label{section:Lorentz}
We shall work in Lorentz  space $\mathbf R^2_1$, namely $\mathbf R^3$ with the bilinear product
$$
(x^0,x^1,x^2)\cdot (y^0,y^1,y^2)= -x^0y^0+x^1y^1+x^2y^2.
$$
A nonzero vector $x\in\mathbf R^2_1$ is called \emph{space-like} if $x\cdot x>0$, \emph{time-like}
if $x\cdot x<0$ and \emph{light-like} if $x\cdot x=0$.

In this model the hyperbolic plane $\mathbf H^2$ and the de Sitter sphere $\mathbf S^2_1$ are
$$
\begin{array}{rcl}
 \mathbf H^2 & = & \{ x\in \mathbf R^2_1\mid x\cdot x= -1, \, x^0>0\}, \\
 \mathbf S^2_ 1 & = & \{ x\in \mathbf R^2_1\mid x\cdot x= 1 \}.
\end{array}
$$
The Riemannian product of $\mathbf H ^2$ and the Lorentz product of $\mathbf S^2_1$ are the restriction to the tangent space of
the bilinear product of  $\mathbf R^2_1$.

 Points in de Sitter sphere  $\mathbf S^2_ 1 $ are viewed as \emph{oriented lines} in hyperbolic planes, cf.\ \cite{Coxeter,Ratcliffe}.
To orient a line is equivalent to chose a unitary normal vector field.
The correspondence between oriented lines and points in $\mathbf S^2_ 1$ is as follows: 
 if $l\subset \mathbf H^2$ is an oriented line, then there is a unique $l^*\in\mathbf S^2_ 1 $ so that 
$$
l= \{ x\in\mathbf H^2\subset \mathbf R^2_1 \mid x\cdot l^*=0\}
$$ 
and the normal vector points to the half-space $\{ x\in\mathbf H^2\subset \mathbf R^2_1 \mid x\cdot l^*>0\}$.

For a point $v\in\mathbf H^2$, the set of oriented lines through $v$ is a geodesic in $\mathbf S^2_1$:
$$
v'=\{l^*\in \mathbf S^2_1\mid l^*\cdot v=0\}.
$$

For a convex hyperbolic polygon we follow the convention that its edges are oriented outwards. In particular if $e_1,\ldots,e_n$ are the edges of a polygon, then its interior is the set
$\{x\in\mathbf H^2\mid x\cdot e_i^*<0, i=1,\ldots,n\}
$.

\begin{Definition}
 Given a convex polygon $p\subset \mathbf H^2$ its \emph{polar} $p^*\subset \mathbf{S}^2_1$  is the set of oriented lines that meet $p$ at precisely one point and the normal vector points outwards.
\end{Definition}

Let  $v_1,\ldots v_n\in\mathbf H^2 $ and $e_1,\ldots,e_n\subset \mathbf H^2$ denote the respective vertices and edges of $p$, 
so that $e_i$ joins $v_{i-1}$ to $v_{i}$ (with $v_0=v_n$).
Then $e_1^*,\ldots,e_n^*$
are the vertices of $p^*$. Let also
 $v_1^*,\ldots,v_n^*$ denote the edges of $p^*$
(each $v_i^*$ is a segment of the geodesic $v_i'\subset \mathbf S^2_1$ dual to $v_i$, defined above). 
If the angle of $p $ at $v_i$ is $\beta _i$, then the length of $v_i^*$ is $\pi-\beta_i$, and
the length of $e_i$ equals the (time-like) angle of $p^*$ at  $e_i^*$.

The following lemma and remarks will be useful in the proof of Theorem~\ref{theorem1}.

\begin{Lemma}
\label{lemma:coplanar}
A \emph{compact} polygon in $\mathbf H^2$ with edges $e_1,\ldots,e_n$ has an inscribed circle if and only if 
$e_1^*,\ldots,e_n^*$ are coplanar in Lorentz  space $ \mathbf R^2_ 1$.
\end{Lemma}

\begin{proof}
The proof requires the following formula, that can be easily proved (see for instance \cite[3.2.8]{Ratcliffe}):
If $p\in\mathbf H^2$ is at oriented distance $d\in\mathbf R$ from the geodesic $l\subset\mathbf H^2$, then
$$
p\cdot l^*=\sinh(d).
$$
 From this formula, one of the implications is easy. For the other implication, assume that $e_1^*,\ldots,e_n^*$ are coplanar, then one needs to check that the Lorentz  
normal to the plane is time-like (i.e.\ is a point in $\mathbf H^2$, and therefore the center of the circle). If it was not time-like, then 
the Lorentz  normal would be either light-like (and the edges would be tangent to a horocycle) or  space-like (and the edges would be equidistant to a
geodesic), but this would contradict the compactness of the polygon. More precisely, given a point  $x\in \mathbf H^2$ and a complete  curve  $\Lambda\subset\mathbf H^2$
with constant principal curvature,
there are at most two geodesics through $x$ and tangent to $\Lambda$. This means that if all the edges of a polygon are tangent to $\Lambda$, then when we follow the ordered edges of the 
polygon, the tangency points are monotonic in $\Lambda$. Hence if $\Lambda$ is not a circle, then the polygon is not closed.
\end{proof}

The group of linear isometries of $\mathbf R^2_1$  is denoted by  $SO(2,1)$. 
Restricting the elements of the identity component of $SO_0(2,1)$ to either
$\mathbf H^2$ or $\mathbf S^2_1$,  $SO_0(2,1)$ is the identity component
of the isometry group for $\mathbf H^2$ and also for 
  $\mathbf S^2_1$. Thus:

\begin{Remark}
\label{remark:Lie}
The Lie algebra
of infinitesimal isometries for either space, $\mathbf R^2_1$,  $\mathbf H^2$ and $\mathbf S^2_1$ is 
$\mathfrak{so}(2,1)$.
\end{Remark}

In $\mathbf R^2_1$ there is a Lorentzian cross product $\boxtimes$, defined by the rule 
$$
(u\boxtimes v)\cdot w=\det (u,v,w),\qquad \forall u,v,w\in\mathbf R^2_1,
$$
where $\det (u,v,w)$ denotes the determinant of the matrix with entries the components of $u, v,w$. In particular 
$(\mathbf R^2_1,\boxtimes)$ is a Lie algebra and we have:

\begin{Remark}
\label{remark:identify} 
The space $\mathbf R^2_1$ equipped with the Lorentz product $\boxtimes $ is a Lie algebra isomorphic to 
$\mathfrak{so}(2,1)$.
\end{Remark}

This is the Lorentzian version of the isomorphism between $\mathfrak{so}(3)$ and $\mathbf  R^3$ equipped with the standard cross product.

\section{The space of polygons with given angles} 
\label{section:spaceofpolygons}

Let $\mathcal P$ denote the space of convex hyperbolic polygons with fixed ordered angles
 $0<\beta_1,\ldots,\beta_n<\pi$. 
We embed it in  $\mathbf R^n$, with coordinates the length of the edges $l_1,\ldots,l_n >0$.

By convexity, the closure $\overline{\mathcal P}$  is obtained by considering additional  polygons with edges of length zero.

To analyze the embedding $\mathcal P\subset\mathbf R ^n$, we make the following construction. Fix a point $v_0\in\mathbf H^2$ and
a  unitary tangent vector $u_0\in T_{v_0}^1\mathbf H^2$.
Consider the polygonal path starting at $v_0$ with direction $u_0$ that consists of $n$ (ordered) segments of lengths
$l_1,l_2,\ldots,l_n$ with (ordered) angles $\beta_1,\beta_2,\ldots,\beta_{n-1}$. 
Namely, at the end of the $i$-th edge, turn left the tangent vector by an angle $\pi-\beta_i$ and continue along the geodesic of this vector
up to length $l_{i+1}$.
 At the end of this path,
consider the tangent unitary vector forming an angle $\beta_n$ (turn left the tangent vector an angle $\pi-\beta_n$), and call 
this vector $F(l_1,l_2,\ldots,l_n)\in T^1\mathbf H^2$.
 This construction defines a map from $\mathbf R^n$ to the unit tangent bundle: 
$$
F:\mathbf R^n\to T^1\mathbf H^2.
$$
With this definition
$$
\overline{\mathcal P} \subseteq   F^{-1}(u_0)\cap \{l_1,\cdots,l_n\geq 0\}.
$$ 
Maybe there  is no equality, because there may exist polygonal paths in $F^{-1}(u_0)$ with self-intersection, but by convexity $\mathcal P$ is \emph{open} in  $F^{-1}(u_0)$.
The map $F$ is analytic and we want to show that it is a submersion. For this we identify $T^1\mathbf H^2$ with the group of orientation-preserving 
hyperbolic isometries, 
by mapping any orientation-preserving isometry $\gamma$ to $\gamma(u_0)\in  T^1\mathbf H^2$. Hence, using also Remark~\ref{remark:identify} we have natural isomorphisms
$$
T_{u_0}(T^1\mathbf H^2)\cong \mathfrak{so}(2,1)\cong \mathbf R^2_1.
$$

\begin{Lemma}
\label{lemma:dF}
 The components of the tangent map $F_*: \mathbf R^n\to\mathbf R^2_1$ at a point of $F^{-1}(u_0)$ are 
$
(e_1^*,\ldots, e_n^*)
$. Namely:
$$
F_*\left(\frac{\partial\phantom{l_i}}{\partial l_i}\right)= e_i^*.
$$
\end{Lemma}

\begin{proof}
Vary $l_i$ while keeping  $l_j$ constant for $j\neq i$. The result on $F$
 is equivalent to composing $F$ 
with  the isometry with axis the geodesic containing $e_i$. By the previous isomorphism of tangent spaces, its derivative
$\frac{\partial F}{\partial l_i}$
 is   $e_i^*$.
\end{proof}

The computation of the tangent space below is Theorem A$_{dS}$ in \cite{Schlenker}.

\begin{Proposition}[\cite{Schlenker}]
 \label{lemma:A}
The subset $ \mathcal P\subset \mathbf R^n$ is a smooth analytic submanifold   of codimension three.
The  tangent subspace is:
$$
T_p\mathcal P=\{ (\dot{l}_1,\ldots,\dot{l}_n)\in \mathbf R^n\mid \sum \dot{l}_i e_i^*=0\}\subset \mathbf R^n,
$$
where $\sum \dot{l}_i e_i^*$ lies in $\mathbf R^2_1$.
\end{Proposition}

\begin{proof}
By Lemma~\ref{lemma:dF}, $F$ is locally a submersion, hence $F^{-1}(u_0)$ is an analytic submanifold of codimension three. In addition
 the tangent space is the kernel of the tangent map of $F$.
\end{proof}

By Proposition~\ref{lemma:A}, the perimeter$$
 \overline {\mathcal P}\to\mathbf R
$$
 is a smooth function. It is also proper and it is bounded below away from zero. 
Thus, to prove Theorem~\ref{theorem1} it suffices to prove the following three lemmas.

\begin{Lemma}
\label{Main Lemma} 
A polygon $p\in \mathcal P$ is a critical point for the perimeter if and only if it has an inscribed circle.
\end{Lemma}

\begin{Lemma}
\label{lemma:boundary} 
At every $p\in \partial\mathcal P=\overline{\mathcal P}\setminus \mathcal P$ 
there exists a deformation to the interior of $\mathcal P$ that decreases strictly the perimeter. 
\end{Lemma}

\begin{Lemma}
 \label{lemma:uniquenes}
There is a unique polygon $p\in \mathcal P$ with an inscribed circle.
\end{Lemma}

\begin{proof}[Proof of Lemma~\ref{Main Lemma}]
 By Lemma~\ref{lemma:coplanar}  one must show that $p\in \mathcal P$ is a critical point of the perimeter if and only if
$e_1^*,\ldots,e^*_n$ are coplanar in  $\mathbf R^2_1$. 
By Proposition~\ref{lemma:A}, a critical point for the perimeter  is determined  by the fact that whenever the vector $(\dot{l}_1,\ldots,\dot{l}_n)\in\mathbf R^n$
satisfies $\sum \dot{l}_i e_i^*=0
$, then $\sum \dot{l}_i =0$. Following \cite{Schlenker} again, consider $A\in  M_{3\times n}(\mathbf R)$ the matrix whose columns are the 
components of the vectors $e_1^*,\ldots,e^*_n$ in $\mathbf R^2_1$. Consider $A'\in  M_{4\times n}(\mathbf R)$ the matrix obtained from $A$ by adding a row with each entry equal to $1$.
A  critical point for the perimeter is characterized by the fact that $A'$ and $A$ have the same kernel, or equivalently
$\operatorname{rank}(A')=\operatorname{rank}(A)=3$. This can be restated by saying that, as a linear map from $\mathbf R^4$ to $\mathbf R^n$, 
the transpose $(A')^t$ has a nontrivial kernel. Finally, notice that there is a correspondence between projective classes of non-zero elements in the kernel of $(A')^t$ and
affine planes  containing $e_1^*,\ldots,e^*_n$. 
\end{proof}

\begin{proof}[Proof of Lemma~\ref{lemma:boundary}]
At a boundary polygon $p\in \partial \mathcal P$ there are some edges of length $0$. Assume first that the collapsed edges are all consecutive: $e_1,\ldots,e_k$
(namely $l_1=\cdots=l_k=0$ and $l_{k+1},\ldots,l_n>0$). Even if $l_1=  \cdots=l_k=0$, the $e_1^*,\ldots,e_k^*\in \mathbf S^2_ 1$ are well determined, 
because $\mathcal P$ is 
defined from the angles.
Moreover, the collapsed vertex $v_n=v_1=\cdots=v_k$ belongs to all $e_n,e_1,\ldots,e_{k+1}$. Thus $e_n^*,e_1^*,\ldots,e_{k+1}^*$ are contained in the 
orthogonal to $v_1$, which is a space-like plane.  By convexity (cf.~Figure~\ref{fig:vectors}), 
$
\forall i=1,\ldots,k,$ there are $a_i,b_i>0$, such that  
$$
e_i^*= a_i e_n^* +b_i e_{k+1}^*.
$$ 
Here it is important that $a_i,b_i>0$, and that by the triangle inequality we have:
$$
1 < a_i+b_i.
$$
Consider a deformation with tangent vector
$$
\dot{l}_1=\cdots=\dot{l}_k=1, \quad \dot{l}_n=-\sum_{i=1}^k a_i, \quad \dot{l}_{k+1}=-\sum_{i=1}^k b_i, 
$$
and $\dot{l}_j=0\textrm{ for }k+1<j<n$. Since it satisfies $\sum_{i=1}^n\dot{l}_i e_i^*=0$, this infinitesimal deformation is tangent 
to a deformation in $F^{-1}(u_0)$. All edge lengths become  positive, and convexity is preserved by the position of the normal vectors
$e_n^*,e_1^*,\ldots, e_{k+1}^*$ (cf.~Figure~\ref{fig:vectors}), thus 
it is a deformation to the interior of $\mathcal P$. The derivative of the perimeter in this direction is
$$
\sum_{i=1}^n\dot{l}_i=\sum_{i=1}^k(1-a_i-b_i)<0.
$$

\begin{figure}
\begin{center}
{\psfrag{en}{$e_n^*$}
\psfrag{e1}{$e_1^*$}
\psfrag{e2}{$e_2^*$}
\psfrag{ek1}{$e_{k+1}^*$}
\psfrag{ek}{$e_k^*$}
\psfrag{eek1}{$e_{k+1}$}
\psfrag{een}{$e_n$}
\psfrag{v}{$v_n=v_1=\cdots=v_k$}
\psfrag{dots}{$\vdots$}
\psfrag{hdots}{$\dots$}
\psfrag{h}{$\mathbf H^2$}
\psfrag{p}{$p$}
\psfrag{vv}{$v_1^{\bot}$}
\includegraphics[height=2.5cm]{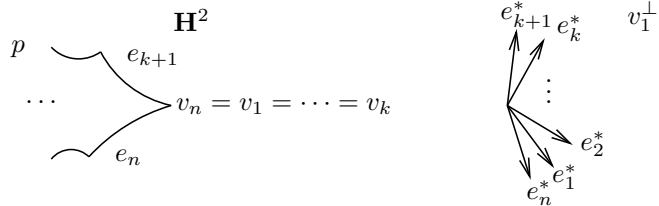}
}
\end{center}
   \caption{The polygon $p$ with collapsed edges, on the left hand side, and the vectors $e_n^*,e_1^*,\ldots,e_{k+1}^*$ in the plane orthogonal to $v_1$, on the right hand side.}\label{fig:vectors}
\end{figure}

The argument applies in general, by grouping the collapsed edges that are consecutive and adding all infinitesimal deformations.
\end{proof}

\begin{proof}[Proof of Lemma~\ref{lemma:uniquenes}]
For each $1\leq i\leq n$ and for $r>0$, consider  the hyperbolic  quadrilateral with ordered angles $\beta_i$, $\pi/2$, $\theta_i(r)$ and $\pi/2$, and side lengths
$r$ for the two edges adjacent to the vertex of angle $\theta_i(r)$, see Figure~\ref{fig:quad}. Polygons in $\mathcal P$ with an inscribed circle are in bijection to the set of $r>0$ satisfying
 $\sum \theta_i(r)= 2\pi$, where $r$ is the radius
of the inscribed circle, and the polygon is obtained by gluing the quadrilaterals along the edges of length $r$. Now the lemma follows from the the following properties:
$\theta_i(r)$ is strictly decreasing  on $r$, $\theta_i(0)= \pi-\beta_i$ (hence $\sum \theta_i(0)> 2\pi$) and $\theta_i(\infty)=0$. 
\end{proof}

\begin{figure}[h]
\begin{center}
{\psfrag{r}{$r$}
\psfrag{b}{$\beta_i$}
\psfrag{t}{$\theta_i(r)$}
\includegraphics[height=2.5cm]{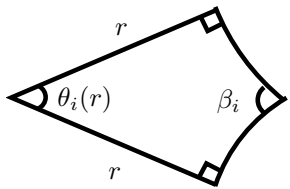}
}
\end{center}
   \caption{The quadrilateral in the proof of Lemma~\ref{lemma:uniquenes}.}\label{fig:quad}
\end{figure}


\paragraph*{Acknowledgements} I am indebted to Igor Rivin and Jean-Marc Schlenker for stimulating conversations.

\begin{footnotesize}
\bibliographystyle{plain}


\noindent \textsc{Departament de Matem\`atiques, Universitat Aut\`onoma de Barcelona,\\  08193 Cerdanyola del Vall\`es (Spain)}

\noindent \emph{porti@mat.uab.cat}
\end{footnotesize}

\end{document}